\documentclass[11pt]{article}
\usepackage[ a4paper, left=20mm, right=20mm, top=20mm, bottom=30mm]{geometry}
\usepackage{graphicx}
\usepackage{float}
\usepackage{tikz}
\usepackage{amsmath}
\usepackage{color}
\usepackage{xcolor}
\usepackage[english]{babel}
\usepackage{amssymb,amsfonts,amsthm} 

\title{Further Results on the Majority Roman Domination in Graphs}
\author{
{\small Azam Sadat Emadi$^1$, Iman Masoumi$^2$, Seyed Reza Musawi$^3$ }
\\
{\small $^{1}$Department of Mathematics, Faculty of Mathematical Sciences, University of Mazandaran, Babolsar, Iran}\\
{\small math\_emadi2000@yahoo.com} \\
{\small $^{2}$ Department of Mathematics, Faculty of Mathematical Sciences, University of Tafresh, Tafresh, Iran}\\{\small  i.masoumi1359@gmail.com}\\
{\small $^{3}$ Faculty of Mathematical Sciences, Shahrood University of Technology, P.O. Box 36199-95161, Shahrood, Iran}\\ 
{\small r\_musawi@yahoo.com} 
}

\newtheorem{theorem}{Theorem}

\newtheorem{example}[theorem]{Example}

\newtheorem{lemma}[theorem]{Lemma}

\newtheorem{observation}[theorem]{Observation}

\newtheorem{proposition}[theorem]{Proposition}

\begin{document}
\maketitle

\begin{abstract}
Let $G=(V(G),E(G))$ be a simple graph of order $n$. A Majority Roman Dominating Function (MRDF) on a graph G is  a function $f: V\rightarrow\{-1, +1, 2\}$ if the sum of its function values over at least half of the closed neighborhoods
of its vertices is at least one, that is, for at least half of the vertices $v\in V$, $f(N[v])=\sum_{x \in N[v]} f(x)\geq 1$. Moreover, every vertex u with $f(u)=-1$ is adjacent to at least one vertex $w$ with $f(w)=2$. The Majority Roman Domination number of a graph $G$, denoted by $\gamma_{MR}(G)$, is the minimum value of  $\sum_{v\in{V(G)}}f(v)$ over all Majority Roman Dominating Function $f$ on $G$.\\
In this paper, we investigate the properties of Majority Roman Domination in graphs and establish some lower and upper bounds for the Majority Roman Domination number for various classes of graphs.
\end{abstract}

\section{Introduction}
Let $G = (V(G),E(G))$ be a simple graph. $V=V(G)$ is the vertex set of $G$ and $E=E(G)$ is the edge set of $G$. 
The \emph{order} of $G$ is equal to the number of vertices of $G$, $\vert V(G) \vert$, the cardinal of $V(G)$. For any vertex $ v\in V$, the \emph{open neighborhood} of $v$ is the set $N(v)=\{u\in V|uv\in E\} $ and the \emph{closed neighborhood} of $v$ is the set $N[v] = N(v)\cup \{v\}$. For a set $S\subseteq V$, the open neighborhood of $S$ is $N (S)=\bigcup_{v\in S}N(v)$ and the closed neighborhood of $S$ is $N[S]=N(S)\cup S$. A set $S\subseteq V$ is a \emph{dominating set} if $N[S] = V$, or equivalently, every vertex in $V\setminus S$ is adjacent to at least one vertex in $S$. The \emph{domination number} of $G$,
  $\gamma(G)$, is the minimum cardinality over all
 dominating sets in $G$. A dominating set with cardinality $\gamma(G)$ is called a \emph{$\gamma(G)$-set}.
 For more information on these parameters, refer to \cite{bl, bz}.
 \\
For a real-valued function $f: V\rightarrow R$, the weight of $f$ is defined as $w(f)=\sum_{v\in{V}}f(v)$ and also for a subset $S\subseteq V$, we define $f(S)=\sum_{v\in{S}} f(v)$. Therefore, $w(f)=f(V)$.
\\
In 1995, Broere et al. introduced the concept of
 \textit{majority domination} \cite{Broere}. A function $f: V\rightarrow \{-1, +1\}$ is called a majority dominating function
on $G$, if $f(N([v]))\geq 1$ for at least half of the vertices in $G$. The majority domination number of $G$, 
denoted by $\gamma_{maj}(G)$, is defined as $\gamma_{maj}(G)=min \{w(f) | f$ is a majority dominating function 
on $G\}$. And also, a function $f: V\rightarrow \{-1, +1\}$ is called a signed majority total dominating function
on $G$, if $f(N(v)) \geq 1$ for at least half of the vertices in $G$. The signed majority total domination number of $G$,  denoted by $\gamma^{t}_{maj}(G)$,
  is defined as $\gamma^{t}_{maj}(G)=min \{w(f) | f$ is a signed majority total dominating function 
  on $G\}$.
More results on the majority domination and signed majority total domination in graphs can be found in \cite{holm, jos, swa, xing, yeh}.
\\
 A Roman dominating function (RDF) on a graph $G=(V, E)$ is a function $f:V\rightarrow \{0, 1, 2\}$ satisfying the condition that every vertex $u$ for which $f(u)=0$ is adjacent to at least one vertex $v$ for which $f(v)=2$. The Roman domination number of $G$, denoted by $\gamma_{R}(G)$, equals the minimum weight of an RDF on $G$ \cite{cok}.
 \\
In 2021, S. Anandha Prabhavathy defined a majority Roman dominating function on a graph $G=(V, E)$ to be a function $f: V\rightarrow \{-1, +1, 2\}$ satisfying two conditions: (i) The sum of its function values over at least half of the closed neighborhood of its vertices is at least one, and (ii) Each vertex $u$ for which $f(u)=-1$ is adjacent to at least one vertex $v$ such that $f(v)=2$ \cite{anan}. The weight of an MRDF is the sum of its function values over all vertices. The Majority Roman domination number of $G$, denoted by $\gamma_{MR}(G)$, is defined as $\gamma_{MR}(G)=min \{w(f) | f$ is a majority Roman dominating function on $G\}$.
    \\
In this paper, motivated by the above concepts, we study the majority Roman domination number in some graphs and establish some bounds for the majority Roman domination number of some graphs.

\section{Majority Roman Domination Number of some graphs}
In this section, we obtain some bounds of the majority Roman domination number for graph $G$. 
Furthermore, we study the majority Roman domination number for various graphs. \\
In fact, If $f$ is an MRDF on a graph $G$, we divide the vertex set of $G$ into three 
disjoint subsets $V_{f,1}$, $V_{f,2}$ and $V_{f,-1}$ as follows:\\
$V_{f,1}=\{v \in V(G): f(v)=1\}$\\
$V_{f,2}=\{v \in V(G): f(v)=2\}$\\
$V_{f,-1}=\{v \in V(G): f(v)=-1\}$\\
Also, we consider the set $N_{f}=\{v \in V(G): f([v])\geqslant1\}$.

\begin{proposition}
Let $G$ be a graph with $n$ vertices, then 
   $\gamma_{MR}(G) \leqslant n$, and equality holds if and only if $G=\overline{K_n}$.
 \end{proposition}
\begin{proof}
It is easy to see that the constant function $f$ with $f(v)=1$ for all $v\in V$, is an MRDF on $G$. Hence, $\gamma_{MR}(G)\leq n$ and for graph $G=\overline{K_n}$ , we have $\gamma_{MR}(G) =n$. Now, assume that for a graph $G$, $\gamma_{MR}(G) =n$. We aim to prove that this graph 
 has no edge. By the contrary, suppose that $G$ has at least one edge. We consider two cases:\\
\textbf{Case 1. }
Let $G$ has a leaf  $v$, that is, $deg(v)=1$, and $w$ is its unique adjacent. We consider an MRDF $f$, defined by  $f(v)=-1$, $f(w)=2$ and $f(x)=1$ for all vertices other than $v$ and $w$.
Then, $\gamma_{MR}(G) \leq f(V)=2+(-1)+(n-2)=n-1$, a contradiction.\\
\textbf{Case 2. }
 Let $G$ has no leaf. Then there are two adjacent vertices $v$ and $w$ with degrees greater than one. Again, the function $f$ defined by $f(v)=-1$, $f(w)=2$ and $f(x)=1$ for all vertices other than $v$ and $w$, is an MRDF on $G$.
Now, we have $\gamma_{MR}(G) \leq f(V)= 2+(-1)+(n-2)=n-1$, a contradiction.\\
Therefore, $G$ have no edge, as required, and $G=\overline{K_n}$ for some $n$.
\end{proof}

We know that for any graph $G$, $\gamma(G)\leq \gamma_{R}(G) \leq 2 \gamma(G)$ \cite{cok}. Here, we obtain similar results about the majority Roman domination number of graphs.

\begin{proposition}\label{02}
  Let $G\ne\overline{K_n}$ be a graph. Then, $2\gamma(G)+1-n\leq \gamma_{MR}(G) \leq 2\gamma(G)$. These bounds are sharp, the lower bound for graph $P_3$, and the upper bound for  $C_3$.
\end{proposition}

\begin{proof}
To prove the lower bound, we consider a function $f$ as an MRDF on $G$.
First, we notice that the set $V_{f,2}\cup V_{f,1}$ is a dominating set of $G$, and since $G\ne\overline{K_n}$, we have $|V_{f,2}|\geq1$.
Now, we can see
$$f(V)=2|V_{f,2}|+|V_{f,1}|-|V_{f,-1}|=3|V_{f,2}|+2|V_{f,1}|-n=2(|V_{f,2}|+|V_{f,1}|)+|V_{f,2}|-n$$
and since $V_{f,2}$ and $V_{f,1}$ are two distinct sets, we have
$$f(V)=2|V_{f,2}\cup V_{f,1}|+|V_{f,2}|-n\geq 2\gamma(G)+1-n$$
that results $\gamma_{MR}(G)\geq 2\gamma(G)+1-n$ as required.

 To prove the upper bound, we construct an MRDF $f$ such that $V_{f,2}$ is a $\gamma$-set of $G$ and $|V_{f,1}|\leqslant|V_{f,-1}|$. First, we partition the vertex set $V$ to three distinct non-empty set $D$, $A$ and $B$ such that $D$ is a $\gamma$-set of $G$ and 
  $0\leqslant|A|-|B|\leqslant1$,
 then we define the function
 $h(v)=\left\{
 \begin{array}{cc}
 2&v\in D\\
 -1&v\in A\\
 1&v\in B
\end{array}.
 \right.$
\\
If $|N_h|\geqslant\dfrac{n}{2}$, we set $f=h$ and the proof is completed.
Otherwise, if $|N_h|<\dfrac{n}{2}$, then we define the new function
 $g(v)=\left\{
 \begin{array}{cc}
 2&v\in D\\
 -1&v\in B\\
 1&v\in A
\end{array}.
 \right.$
\\
Now, we have $|N_g|\geqslant\dfrac{n}{2}$.
In fact, if $f(x)\leq0$ for some vertex $x$, we have $g(x)\geq4$.
If $|A|=|B|$, then we set $f=g$ and the proof is completed and if $|A|=|B|+1$, then we choose an arbitrary vertex $x\in A$ and we change the value function $g(x)$ from $1$ to $-1$.The inequality $|N_g|\geqslant\dfrac{n}{2}$ still holds.
In each case, we have $|V_{f,2}|=|D|$, $|V_{f,-1}|\geqslant|V_{f,1}|$ and $f(V)\leqslant2\gamma(G)$.
\end{proof}

The join of two graphs $G_1$ and $G_2$, denoted by $G_1 \vee G_2$, is a graph with vertex set $V(G_1) \cup V(G_2)$ and edge set $E(G_1) \cup E(G_2) \cup \{xy : x \in V(G_1), y \in V(G_2)\}$. For example, $K_1 \vee K_2=K_3$, $K_2 \vee K_2=K_4$, $\overline{K_r} \vee \overline{K_s}=K_{r,s}$.

\begin{observation}\label{o1}
Let $G$ and $H$ be two graphs such that $\gamma_{MR}(G) \geq 0$ and $\gamma_{MR}(H) \geq 0$. Then, $\gamma_{MR}(G \vee H) \leq \gamma_{MR}(G) + \gamma_{MR}(H)$.   
\end{observation}
\begin{proof} Let the function $f$ be an MRDF on $G$. So, we have $f(N[v_i])\geq 1$, for at least half of vertices in graph $G$. And also, let the function $g$ be an MRDF on $H$. So, we have $g(N[u_i])\geq 1$, for at least half of vertices in graph $H$. We consider the function $k$ on vertices of $G \vee H$ such that:
$k(v_i)=\left\{
\begin{array}{cc}
f(v_i)&v_i \in V(G) \\
g(v_i)& v_i\in V(H)\\
\end{array}.
\right.$
 Now, it can be easily shown that the funktion $k$ is a MRDF on $G \vee H$.
\end{proof}

For two graphs $G=P_2$ and $H=K_1$, we have $\gamma_{MR}(G \vee H) =\gamma_{MR}(C_3)= 2=\gamma_{MR}(G) + \gamma_{MR}(H)$, but
$\gamma_{MR}(G \vee G) =\gamma_{MR}(K_4)= 1<2=\gamma_{MR}(G) + \gamma_{MR}(G)$.\\

The join $K_1 \vee P_{n-1}$, denoted by $F_n$, is called as a fan graph, and the join $K_1 \vee C_{n-1}$, denoted by $W_n$, is called as a wheel graph. 

We prove  the next result on the structure of all $\gamma_{MR}$-functions of $F_n$ and $W_n$.
 
\begin{lemma}\label{lem01}
For graphs $F_n$ and $W_n$, if $f$ be a $\gamma_{MR}$-function, then we have $|V_{f,2}|=1$, and if $n\geqslant5$, then the unique element of $V_{f,2}$ has the degree $n-1$.
\end{lemma} 

\begin{proof}
Let $V(F_n) =V(W_n)= \{v_0, v_1, v_2, \cdots, v_{n-1} \}$, $E(F_n)= \{v_1v_2, v_2v_3, \cdots, v_{n-2}v_{n-1}\} \cup \{v_0v_i : 1 \leq i\leq {n-1}\}$, and
$E(W_n)=E(F_n)\cup\{v_1v_{n-1}\}$.

If $0<r<s<n-1$ and $v_r,v_s\in V_{f,2}$, then function 
$g(v_i)=\left\{
\begin{array}{cc}
f(v_r)&i=0\\
f(v_0)&i=r\\
1&i=s\\
f(v_i)&\text{otherwise}
\end{array}
\right.$
is an MRDF such that $g(V)<f(V)$, a contradiction.

Now, we consider $f(v_0)=f(v_r)=2$, where $0<r<n-1$. if we set $f(v_r)=1$, $f$ still remains an MRDF with less weight, a contradiction. 

Hence in each $\gamma_{MR}$-function of $F_n$ and $W_n$, we have $|V_{f,2}|=1$. If $n=4$,then $V_{f,2}$ can be any of four one element subset of $V$, while for $n\geqslant5$ we must have $V_{f,2}=\{v_0\}$ and $deg(v_0)=n-1$.

\end{proof}

It is proved for wheel graphs that:

\begin{theorem}\cite{anan}\label{wheel}
For any wheel graph $W_n$, on $n \geq 4$ vertices, we have:
$$\gamma_{MR}(W_n)=2 \lceil \dfrac{n}{6} \rceil -n +3$$
\end{theorem}

Now, we claim that the same result holds for fan graphs.
\begin{proposition}
For every fan graph
 $F_n=K_1 \vee P_{n-1}$, $n \geq 4$, we have:
$$\gamma_{MR}(F_n)=\gamma_{MR}(W_n)$$
\end{proposition}
\begin{proof}
Let $V(F_n) = \{v_0, v_1, v_2, \cdots, v_{n-1} \}$ and $E(F_n)= \{v_1v_2, v_2v_3, \cdots, v_{n-2}v_{n-1}\} \cup \{v_0v_i : 1 \leq i\leq {n-1}\}$. Let $f$ be an MRDF on $F_n$. By Lemma \ref{lem01}, we know that $ V_{f,2}=\{v_0\}$,  if $v_i\in N_f$, for some $2\leqslant i\leqslant n-2$, then $\{v_{i-1},v_i, v_{i+1}\}\cup V_{f,1}\ne\emptyset$. On the other hand, 
  if $v_i\in V_{f,1}$, for some $2\leqslant i\leqslant n-2$, then $\{v_{i-1},v_i, v_{i+1}\}\subseteq N_f$. 
  Hence, each vertex in $V_{f,1}$ adds a maximum of 3 vertices to $N_f$. To establish the condition that $|N_f|\geqslant\dfrac{n}{2}$, we must have $3\times |V_{f,1}|\geqslant \dfrac{n}{2}$, that is, $|V_{f,1}|\geqslant \dfrac{n}{6}$
 \\
 Now, we have $\gamma_{MR}(F_n) = 2\vert V_{f,2} \vert +\vert V_{f,1} \vert -\vert V_{f,-1} \vert\geqslant 2\times 1 +\lceil \dfrac{n}{6}\rceil - (n-\lceil \dfrac{n}{6}\rceil-1)\geqslant 2\lceil \dfrac{n}{6}\rceil -n +3 $.
 \\
 Hence, we get the lower bound $\gamma_{MR}(F_n)\geqslant 2\lceil \dfrac{n}{6}\rceil -n +3 $.

 On the other hand, we define the function $g: V(F_n)\rightarrow \{-1, 1, 2\}$ as follows.
 
$$g(v_i)=\left\{
\begin{array}{ccl}
2&&i=0\\
1&&i=3k,\; k=1,2,\cdots, \lceil \dfrac{n}{6}\rceil \\
-1&&\text{otherwise}
\end{array}
\right.$$

It is easy to see that for $n=4$ we have $N_g=\{v_0, v_2, v_3\}$ and for
every $n \geqslant 5$, we have $N_g=\{v_2, v_3,\cdots, v_{3\lceil \dfrac{n}{6}\rceil+1}\}$, and so the function $g$ is an MRDF on graph $F_n$ with $G(V)=2\lceil \dfrac{n}{6}\rceil -n +3 $.

Therefore, we have  $\gamma_{MR}(F_n)=2\lceil \dfrac{n}{6}\rceil -n +3= \gamma_{MR}(W_n)$. Figure \ref{fig1} depicts it for the cases $n=4, 5,6$.
\end{proof}

\begin{figure}[h]\label{01}
\begin{center}
\begin{tikzpicture}
\node at (0,0){
\begin{tikzpicture}
\node at (0,1) [label=above:$2$]{};
\node at (-1,0) [label=below:$-1$]{};
\node at (0,0) [label=below:$-1$]{};
\node at (1,0) [label=below:$1$]{};

\node at (4.5,1) [label=above:$2$]{};
\node at (3,0) [label=below:$-1$]{};
\node at (4,0) [label=below:$-1$]{};
\node at (5,0) [label=below:$1$]{};
\node at (6,0) [label=below:$-1$]{};

\node at (10,1) [label=above:$2$]{};
\node at (8,0) [label=below:$-1$]{};
\node at (9,0) [label=below:$-1$]{};
\node at (10,0) [label=below:$1$]{};
\node at (11,0) [label=below:$-1$]{};
\node at (12,0) [label=below:$-1$]{};

\draw(0,0)--(1,0);\draw(0,0)--(-1,0);\draw(0,0)--(0,1);\draw(0,1)--(-1,0);\draw(0,1)--(1,0);

\draw(5,0)--(6,0);\draw(3,0)--(4,0);\draw(4,0)--(5,0);\draw(4.5,1)--(5,0);\draw(4.5,1)--(6,0);\draw(4.5,1)--(3,0);\draw(4.5,1)--(4,0);

\draw(12,0)--(11,0);\draw(9,0)--(10,0);\draw(10,0)--(11,0);\draw(8,0)--(9,0);\draw(10,1)--(12,0);\draw(10,1)--(8,0);\draw(10,1)--(9,0);\draw(10,1)--(11,0);\draw(10,1)--(10,0);

\foreach \x in {  -1, 0, 1, 3, 4, 5, 6, 8, 9, 10, 11, 12 }
\foreach \y in { 0 } \filldraw (\x,\y)circle(2pt);

\foreach \x in { 0, 4.5, 10}
\foreach \y in { 1 } \filldraw (\x,\y)circle(2pt);

\end{tikzpicture}
};
\end{tikzpicture}
\end{center}
\caption{Majority Roman domination labeling on $F_4$, $F_5$ and $F_6$}
\label{fig1}
\end{figure}

For two graphs $G$ and $H$, we define the Cartesian product of $G$ and $H$ to be the graph $G \square H$ with vertices $\{(u, v)| u \in V(G), ~v \in V(H)\}$. Two vertices $(u_1, v_1)$ and $(u_2, v_2)$ are adjacent in $G \square H$ if and only if one of the following is true: $u_1=u_2$ and $v_1$ is adjacent to $v_2$ in $H$; or $v_1=v_2$ and $u_1$ is adjacent to $u_2$ in $G$. If $G=P_m$ and $H=P_n$, then the Cartesian product $G \square H$ is called the $m \times n$ grid graph and is denoted $G_{m, n}$.
\begin{figure}[h]\label{11}
\begin{center}
\begin{tikzpicture}
\node at (0,0){
\begin{tikzpicture}
\draw(0,0)rectangle(0,1);
\foreach \x in {0}
\foreach \y in {0,1} \filldraw (\x,\y)circle(2pt);
\node at (0,-.3){$2$};\node at (0,1.3){$-1$};
\end{tikzpicture}
};
\node at (2.5,0){
\begin{tikzpicture}
\draw(0,0)rectangle(1,1);
\foreach \x in {0,1}
\foreach \y in {0,1} \filldraw (\x,\y)circle(2pt);
\node at (0,-.3){$1$};\node at (0,1.3){$-1$};
\node at (1,-.3){$-1$};\node at (1,1.3){$2$};
\end{tikzpicture}
};
\node at (6,0){
\begin{tikzpicture}
\draw(0,0)rectangle(2,1);
\draw(1,0)--(1,1);
\foreach \x in {0,1,2}
\foreach \y in {0,1} \filldraw (\x,\y)circle(2pt);
\node at (0,-.3){$1$};\node at (0,1.3){$-1$};
\node at (1,-.3){$-1$};\node at (1,1.3){$2$};
\node at (2,-.3){$1$};\node at (2,1.3){$-1$};
\end{tikzpicture}
};
\node at (10.5,0){
\begin{tikzpicture}
\draw(0,0)rectangle(3,1);
\draw(1,0)--(1,1);\draw(2,0)--(2,1);
\foreach \x in {0,1,2,3}
\foreach \y in {0,1} \filldraw (\x,\y)circle(2pt);
\node at (0,-.3){$1$};\node at (0,1.3){$-1$};
\node at (1,-.3){$-1$};\node at (1,1.3){$2$};
\node at (2,-.3){$-1$};\node at (2,1.3){$-1$};
\node at (3,-.3){$2$};\node at (3,1.3){$1$};
\end{tikzpicture}
};
\end{tikzpicture}
\end{center}
\caption{A majority Roman domination labeling on $G_{2,1}$, $G_{2,2}$, $G_{2,3}$ and $G_{2,4}$}
\label{fig2}
\end{figure}

\begin{theorem}
For the $2 \times n$ grid graph $G_{2, n}$, we have,
 $\gamma_{MR}(G_{2,4})=2$ and $\gamma_{MR}(G_{2,4})\leq 1$, $n\neq 4$.
\end{theorem}

\begin{proof} 
First of all, we prove $\gamma_{MR}(G_{2, 4})=2$. Suppose that $f$ is a $\gamma_{MR}$-function  of $G_{2, 4}$. We show that $f(V)\geqslant 2$. Note that $G$ has 8 vertices with at most 3 neighbors.\\
(i) $|V_{f,2}|=1$. we know $|V_{f,-1}|\leqslant3$ and so $f(V)\geqslant 2-3+4=3$. 
\\
(ii) $|V_{f,2}|=2$. we know $|V_{f,-1}|\leqslant5$. If  $|V_{f,-1}|\leqslant4$, then $f(V)\geqslant 2\times2-4+2=2$.
\\
If $|V_{f,-1}|=5$, then by symmetry, there is a unique definition for $f$ to satisfy for the adjacency condition for MRDFs, but we see that $|N_f|=2$, a contradiction.\\
(iii) $|V_{f,2}|=3$. we know $|V_{f,-1}|\leqslant5$. If  $|V_{f,-1}|\leqslant4$, then $f(V)\geqslant 3\times2-4+1=3$.
\\
If $|V_{f,-1}|=5$, then by symmetry, there is three definition for $f$ to satisfy for the adjacency condition for MRDFs, but we see that $|N_f|\leq3$, a contradiction.\\
(iv) $|V_{f,2}|\geqslant4$. It is clear that $f(V)\geqslant4\times2-4=4$.\\
Therefore, $f(V)=\gamma_{MR}(G_{2, 4})\geq 2$. On the other hand, the MRDF on $G_{2, 4}$, in Figure \ref{fig2}, shows that $\gamma_{MR}(G_{2, 4})=2$.

We show that $\gamma_{MR}(G_{2, n}) \leq 1$, for $n\neq 4$.\\
We set $E(G_{2, n})=\{v_{1, 1}, v_{1,2}, \ldots, v_{1,n}, v_{2,1}, v_{2,2}, \dots, v_{2,n}\}$ and we define four different functions $f$, $g$, $h$ and $k$ as the majority Roman dominating functions on graph $G_{2, n}$ for $n=4k, ~k\geq 2$, $n=4k+1$, $n=4k+2$ and $n=4k+3$, respectively.

Let $f: V(G_{2, n}) \rightarrow \{1, -1, 2\}$ be defined as follows. Then $f$ is an MRDF on $ G_{2, n}$ for $n=4k$, where $k \geq 2$.

$$
f(v_{i, j})=\left\{
\begin{array}{ccl}
1&& i=1,\; j=4k^{'},~ 1\leq k^{'} \leq k-1\\
2&& i=1,\; j=4k^{'}+3,~ 1\leq k^{'} \leq k-1\\
2&& i=2,\; j=4k^{'}+1 , ~1 \leq k^{'} \leq k-1\\
2&&i=1,j=2\;\textnormal{or}\; i=2,j=2,n-1\\
-1&&\textnormal{ otherwise} 
\end{array}\right.
$$
Figure \ref{fig3} illustrates it for the case $n=12$.
\begin{figure}[h]\label{12}
\begin{center}
\begin{tikzpicture}
\node at (0,0){
\begin{tikzpicture}
\draw(0,0)rectangle(11,1);
\node at (0,1.3){$-1$};\node at (1,1.3){$2$};\node at (2,1.3){$-1$};
\node at (3,1.3){$1$};\node at (4,1.3){$-1$};\node at (5,1.3){$-1$};
\node at (6,1.3){$2$};\node at (7,1.3){$1$};\node at (8,1.3){$-1$};
\node at (9,1.3){$-1$};\node at (10,1.3){$2$};\node at (11,1.3){$-1$};
\node at (0,-.3){$-1$};\node at (1,-.3){$2$};\node at (2,-.3){$-1$};
\node at (3,-.3){$-1$};\node at (4,-.3){$2$};\node at (5,-.3){$-1$};
\node at (6,-.3){$-1$};\node at (7,-.3){$-1$};\node at (8,-.3){$2$};
\node at (9,-.3){$-1$};\node at (10,-.3){$2$};\node at (11,-.3){$-1$};
\draw(1,1)--(1,0);\draw(2,1)--(2,0);\draw(3,1)--(3,0);\draw(4,1)--(4,0);\draw(5,1)--(5,0);\draw(6,1)--(6,0);\draw(7,1)--(7,0);\draw(8,1)--(8,0);\draw(9,1)--(9,0);\draw(10,1)--(10,0);
\foreach \x in {0,1,2,3,4,5,6,7,8,9,10,11}
\foreach \y in {0,1} \filldraw (\x,\y)circle(2pt);
\end{tikzpicture}
};
\end{tikzpicture}
\end{center}
\caption{Majority Roman domination labeling on $G_{2,12}$}
\label{fig3}
\end{figure}

When $n=4k+1$, we define the function $g: V(G_{2, n})\rightarrow \{1, -1, 2\}$ as follows. $g$ is an MRDF on $G_{2, n}$.

$$
g(v_{i, j})=\left\{
\begin{array}{ccl}
1&& i=2,\; j=1\\
1&& i=1,\; j=4k^{'},~ 1\leq k^{'} \leq k-1\\
2&& i=1,\; j=4k^{'}+2,~ 0\leq k^{'} \leq k-1\\
2&& i=2,\; j=4k^{'} , ~1 \leq k^{'} \leq k\\
2&&i=1,j=n-1\\
-1&&\textnormal{ otherwise} 
\end{array}\right.
$$

Figure \ref{fig4} illustrates it for the case $n=13$.
 
\begin{figure}[h]\label{13}
\begin{center}
\begin{tikzpicture}
\node at (0,0){
\begin{tikzpicture}
\draw(0,0)rectangle(12,1);
\node at (0,1.3){$-1$};\node at (1,1.3){$2$};\node at (2,1.3){$-1$};
\node at (3,1.3){$1$};\node at (4,1.3){$-1$};\node at (5,1.3){$2$};
\node at (6,1.3){$-1$};\node at (7,1.3){$1$};\node at (8,1.3){$-1$};
\node at (9,1.3){$2$};\node at (10,1.3){$-1$};\node at (11,1.3){$2$};
\node at (12,1.3){$-1$};\node at (0,-.3){$1$};\node at (1,-.3){$-1$};
\node at (2,-.3){$-1$};\node at (3,-.3){$2$};\node at (4,-.3){$-1$};
\node at (5,-.3){$-1$};\node at (6,-.3){$-1$};\node at (7,-.3){$2$};
\node at (8,-.3){$-1$};\node at (9,-.3){$-1$};\node at (10,-.3){$-1$};
\node at (11,-.3){$2$};\node at (12,-.3){$-1$};
\draw(1,1)--(1,0);\draw(2,1)--(2,0);\draw(3,1)--(3,0);\draw(4,1)--(4,0);\draw(5,1)--(5,0);\draw(6,1)--(6,0);\draw(7,1)--(7,0);\draw(8,1)--(8,0);\draw(9,1)--(9,0);\draw(10,1)--(10,0);\draw(11,1)--(11,0);
\foreach \x in {0,1,2,3,4,5,6,7,8,9,10,11,12}
\foreach \y in {0,1} \filldraw (\x,\y)circle(2pt);
\end{tikzpicture}
};
\end{tikzpicture}
\end{center}
\caption{Majority Roman domination labeling on $G_{2,13}$}
\label{fig4}
\end{figure}
When $n=4k+2$, we define the function 

$$
h(v_{i, j})=\left\{
\begin{array}{ccl}
1&& i=2,\; j=1\\
1&& i=1,\; j=4k^{'},~ 1\leq k^{'} \leq k\\
2&& i=1,\; j=4k^{'}+2,~ 0\leq k^{'} \leq k\\
2&& i=2,\; j=4k^{'} , ~1 \leq k^{'} \leq k\\
-1&&\textnormal{ otherwise} 
\end{array}\right.
$$
as an MRDF on $G_{2, n}$.

Figure \ref{fig5} illustrates it for the case $n=14$.
\begin{figure}[h]\label{14}
\begin{center}
\begin{tikzpicture}
\node at (0,0){
\begin{tikzpicture}
\draw(0,0)rectangle(13,1);
\node at (0,1.3){$-1$};\node at (1,1.3){$2$};\node at (2,1.3){$-1$};
\node at (3,1.3){$1$};\node at (4,1.3){$-1$};\node at (5,1.3){$2$};
\node at (6,1.3){$-1$};\node at (7,1.3){$1$};\node at (8,1.3){$-1$};
\node at (9,1.3){$2$};\node at (10,1.3){$-1$};\node at (11,1.3){$1$};
\node at (12,1.3){$-1$};\node at (13,1.3){$2$};\node at (0,-.3){$1$};
\node at (1,-.3){$-1$};\node at (2,-.3){$-1$};\node at (3,-.3){$2$};
\node at (4,-.3){$-1$};\node at (5,-.3){$-1$};\node at (6,-.3){$-1$};
\node at (7,-.3){$2$};\node at (8,-.3){$-1$};\node at (9,-.3){$-1$};
\node at (10,-.3){$-1$};\node at (11,-.3){$2$};\node at (12,-.3){$-1$};
\node at (13,-.3){$-1$};
\draw(1,1)--(1,0);\draw(2,1)--(2,0);\draw(3,1)--(3,0);\draw(4,1)--(4,0);\draw(5,1)--(5,0);\draw(6,1)--(6,0);\draw(7,1)--(7,0);\draw(8,1)--(8,0);\draw(9,1)--(9,0);\draw(10,1)--(10,0);\draw(11,1)--(11,0);\draw(12,1)--(12,0);
\foreach \x in {0,1,2,3,4,5,6,7,8,9,10,11,12,13}
\foreach \y in {0,1} \filldraw (\x,\y)circle(2pt);
\end{tikzpicture}
};
\end{tikzpicture}
\end{center}
\caption{Majority Roman domination labeling on $G_{2,14}$}
\label{fig5}
\end{figure}

Finally, we have the following function $h$ as an MRDF on $G_{2, n}$, where $n=4k+3$.

$$
h(v_{i, j})=\left\{
\begin{array}{ccl}
1&& i=2,\; j=1, 4k+3\\
1&& i=1,\; j=4k^{'},~ 1\leq k^{'} \leq k\\
2&& i=1,\; j=4k^{'}+2,~ 0\leq k^{'} \leq k\\
2&& i=2,\; j=4k^{'} , ~1 \leq k^{'} \leq k\\
-1&&\textnormal{ otherwise} 
\end{array}\right.
$$

Figure \ref{fig6} illustrates it for the case $n=15$.

\begin{figure}[h]\label{15}
\begin{center}
\begin{tikzpicture}
\node at (0,0){
\begin{tikzpicture}
\draw(0,0)rectangle(14,1);
\node at (0,1.3){$-1$};\node at (1,1.3){$2$};\node at (2,1.3){$-1$};
\node at (3,1.3){$1$};\node at (4,1.3){$-1$};\node at (5,1.3){$2$};
\node at (6,1.3){$-1$};\node at (7,1.3){$1$};\node at (8,1.3){$-1$};
\node at (9,1.3){$2$};\node at (10,1.3){$-1$};\node at (11,1.3){$1$};
\node at (12,1.3){$-1$};\node at (13,1.3){$2$};\node at (14,1.3){$-1$};
\node at (0,-.3){$1$};\node at (1,-.3){$-1$};\node at (2,-.3){$-1$};
\node at (3,-.3){$2$};\node at (4,-.3){$-1$};\node at (5,-.3){$-1$};
\node at (6,-.3){$-1$};\node at (7,-.3){$2$};\node at (8,-.3){$-1$};
\node at (9,-.3){$-1$};\node at (10,-.3){$-1$};\node at (11,-.3){$2$};
\node at (12,-.3){$-1$};\node at (13,-.3){$-1$};\node at (14,-.3){$1$};
\foreach \x in {0,1,...,14}{\draw(\x,1)--(\x,0);
\foreach \y in {0,1} \filldraw (\x,\y)circle(2pt);}
\end{tikzpicture}
};
\end{tikzpicture}
\end{center}
\caption{Majority Roman domination labeling on $G_{2,15}$}
\label{fig6}
\end{figure}
By these four cases, it is straightforward to verify that  $\gamma_{MR}(G_{2, n})\leq 1$ for all positive integers $n\neq 4$, as illustrated in Figures \ref{fig2}, \ref{fig3}, \ref{fig4}, \ref{fig5} and \ref{fig6}.
\end{proof}
In \cite{anan} is proved that for the path graphs with $n\geqslant2$ , we have 
$\gamma_{MR}(P_n)=\lfloor \dfrac{n}{2}\rfloor-\lfloor \dfrac{n+1}{4}\rfloor$
and for the cycle graphs with $n\geqslant3$ , we have 
$\gamma_{MR}(C_n)=\lfloor \dfrac{n+1}{2}\rfloor-\lfloor \dfrac{n}{4}\rfloor$.\\
Here, we will consider the complement of this two important family of graphs. With case calculations for $n\leq12$, we have:
\[
\begin{array}{|c|ccccccccccc|}
\hline
n&2&3&4&5&6&7&8&9&10&11&12\\
\hline
\gamma_{MR}(\overline{P_n})&2&2&1&1&0&1&0&0&-1&0&-1\\
\hline
\gamma_{MR}(\overline{C_n})&&3&2&2&1&1&0&-1&-1&-1&-1\\
\hline
\end{array}
\]

\begin{proposition}\label{51}
Let $n\geq 12$.\\
\textnormal{(i)} For the complement of the path graph $P_n$, we have $\gamma_{MR}(\overline{P_n})=-1$.\\
\textnormal{(ii)} For the complement of the cycle graph $C_n$, we have $\gamma_{MR}(\overline{C_n})=-1$.
\end{proposition}

\begin{proof}  
Since the proofs of \textnormal{(i)} and \textnormal{(ii)}are completely similar, we only prove the first part .\\
Let $E(P_n)=\{v_1v_2, v_2v_3, \cdots, v_{n-1}v_n\}$ and 
$f$ be a $\gamma_{MR}(\overline{P_n})$-function. By the definition of majority Roman dominating function, there exists at least one vertex 
$v_i$ with at most two non-adjacent vertices $x$ and $y$, and $f(N[v_i]) \geq1$. Hence, $\gamma_{MR}(\overline{P_n})= f(V)=f(N[v_i])+f(x)+f(y)\geqslant 1-1-1= -1$.\\
On the other hand, if $n\geqslant12$ and $2n+1=3q+r$, where  $r\in\{0,1,2\}$, we define the  function
 $$
g_n(v_i)=\left\{
\begin{array}{cl}
-1&1\leqslant i \leqslant q\\
2&q+1\leqslant i \leqslant n-r\\
1&\text{otherwise}
\end{array}
\right.
$$
such that for any $i=2,3,\cdots,q-1$, we have  $g_n([v_i])=1$,and so $|N_g|=q-2\geqslant \dfrac{n}{2}$, that is, $g_n$ is an MRDF. Also, we have $f(V)=-q+2(n-r-q)+r =2n-3q-r=-1$. Hence, $\gamma_{MR}(\overline{P_n})=-1$ and $g_n$ is a $\gamma_{MR}(\overline{P_n})$-function.

\end{proof}

%
%
%
%

\begin{proposition}
Let $n\geq 3$ and $M$ be a perfect matching in the complete graph $K_{2n}$. If  $G=K_{2n}-M$, then  $\gamma_{MR}(G)= 0$.
\end{proposition}

\begin{proof}
Let $V(K_{2n})=\{v_1, v_2, \cdots, v_{2n}\}$ and $M=\{v_1v_2, v_3v_4, \cdots, v_{2n-1}v_{2n}\}$. Let $f$ is an MRDF on $G$, according to the definition of majority Roman dominating function, there is at least one vertex $v_i$ such that $f(N[v_i])\geq 1$.\\
Assume that $v_j$ is the unique non-adjacent vertex to $v_i$, therefore, $f(V)=f(N[v_i]) +f(v_{j})\geqslant 1+ (-1) =0$.\\
 On the other hand, we can define the function $g: V\rightarrow \{-1, 1, 2\}$ by
$$
g(v_i)=\left\{\begin{array}{ccc}
-1&&2\leq i \leq n +2\\
2&&i= 1, ~2n\\
1&&\text{otherwise}
\end{array}\right.
$$
such that $g(V)=-(n+1)+2\times2+(n-3)=0$.
Since  $g(N[v_i])\geq 1$, for $i\in\{3,4,\cdots,n+1\}\cup\{1,2(n-\lfloor\dfrac{n}{2}\rfloor+1)\}$, we have $|N_g|=n+1$ and so $g$  is  an $MRDF$ on $G$. Hence, $\gamma_{MR}(G)= 0$ and $g$ is a $\gamma_{MR}(G)$-function.
\end{proof}

\section{Relationship between the Majority Roman Domination number $\gamma_{MR}$, Domination number $\gamma$ and Independence number $\beta_0$ of Trees.}

In this section, we obtain some bounds for the majority Roman domination number of trees in terms of $\gamma(T)$ and $\beta_0(T)$. 

\begin{theorem}\label{ore}\cite{bl}
(Ore) If a graph $G$ has no isolated vertices, then $\gamma(G) \leq \dfrac{n}{2}$. 
\end{theorem}

\begin{theorem}\label{st}\cite{anan}
For any star $S_n=K_{1,n-1}$ on $n\geq 2$ vertices, $\gamma_{MR}(S_n)=3-n$.
\end{theorem}

\begin{theorem}\label{the11}
Let $T$ be a tree of order $n\geq 2$, if there exists a $\gamma$-set, say $D$, such that $V\setminus D$ is an  independent set. Then $\gamma_{MR}(T)\leq 3\gamma(T) -n$. This upper bound is sharp by the path graph $P_5$.
\end{theorem}

\begin{proof}
Let $D$ is a $\gamma(T)$-set. By Theorem \ref{ore}, we know that $|D| \leq \dfrac{n}{2}$.
We define the  function 
\[
f(u)=\left\{\begin{array}{ccc}
2&&u\in D\\-1&&u\notin D
\end{array}•
\right.
\]
on $T$. If $u\in V\setminus D$, then $\emptyset\ne N(u)\subseteq D$ and so $f([u])\geqslant -1+2=1$. Since $|V\setminus D|\geqslant \dfrac{n}{2}$, $f$ is an MRDF on $T$. Finally, we have  $\gamma_{MR}(T)\leqslant f(V)=2\gamma(T)-(n-\gamma(T))=3\gamma(T)-n$.
\end{proof}

The following example shows that there are some infinite families of trees that satisfy the equality $\gamma_{MR}(T)= 3\gamma(T) -n$.

\begin{example}
Let $S_{a,b}$ be the double star on $a+b$ vertices with central vertices $u$ and $v$ with $deg(u)=a\geq 3$ and $deg(v)=b\geq 3$. It is clear that $\gamma_{MR}(S_{a,b})=3\gamma(S_{a,b}) -n$.
Double star graph $S_{5,6}$ is shown  in Figure \ref{fig7}.
\end{example}

\begin{figure}[h]
\begin{center}
\begin{tikzpicture}
\draw[line width=1](2,0)--(4,0);
\node at (2,-.3){$u$};\node at (4,-.3){$v$};
\foreach \x in {2,4} {
\node at (\x,.5){$2$};
 \filldraw (\x,0)circle(2pt); 
};
\foreach \y in {-1.4,-.5,.5,1.4} {
\node at (0,\y){$-1$};
 \filldraw (.5,\y)circle(2pt); 
  \draw[line width=1](2,0)--(0.5,\y);
};
\foreach \y in {-1.4,-.7,0,.7,1.4} {
\node at (6,\y){$-1$};
 \filldraw (5.5,\y)circle(2pt);
 \draw[line width=1](4,0)--(5.5,\y); 
};
\end{tikzpicture}
\end{center}
\caption{$\gamma_{MR}(S_{5,6})=-5$}
\label{fig7}
\end{figure}
In order to characterize all trees attaining upper bound in Theorem \ref{the11}, we define the family $\mathcal{T}$ of trees consisting all stars $S_p,\,P\geq3$ and all trees $T$ that can be obtained from a sequence $T_1, T_2, \cdots, T_k $ in $\mathcal{T}$ such that $T_1$ is a star $S_{p}$ and if $2\leq i\leq k$, then  $T_{i}$ can be obtained  from $T_{i-1}$ by the following Operetion.
\\
\textbf{Operetion $\Phi$:} If $T^*\in\mathcal{T}$, attach a copy of a star $S_{q}\in\mathcal{T}$, by joining its center to a support in $T^*$.

\begin{theorem}
Let $G$ be a graph with order $n\geq 2$ and maximum degree $\Delta$. Then
 $$\gamma_{MR}(G) \geq  \dfrac{n(2-\Delta)}{\Delta+1}$$.
\end{theorem}

\begin{proof}
Let the function $f$ be an MRDF on $G$. Then $f(V)= \vert V_{f,1} \vert  + 2 \vert V_{f,2}\vert - \vert V_{f,-1}\vert$. 

On the other hand, we know that $n = \vert V_{f,1}\vert  + \vert V_{f,2}\vert + \vert V_{f,-1} \vert$.

No, we have $f(V)+n= 2\vert V_{f,1} \vert  + 3 \vert V_{f,2}\vert $ and by multiplying both sides of the equality by $\Delta+1$, we have

\begin{align*}
(\Delta+1)(f(V)+n)&= 2(\Delta+1)\vert V_{f,1} \vert  + 3(\Delta+1) \vert V_{f,2}\vert \\
&= 2(\Delta+1)\vert V_{f,1} \vert  + 3(\Delta\vert V_{f,2}\vert)+3 \vert V_{f,2}\vert \\
&\geqslant  2(\Delta+1)\vert V_{f,1} \vert  + 3\vert V_{f,-1}\vert+3 \vert V_{f,2}\vert \\
&\geqslant  3(\vert V_{f,1} \vert  +3\vert V_{f,-1}\vert+ \vert V_{f,2}\vert) \\
&=3n
\end{align*}

Now, We have $(\Delta+1)(f(V)+n)\geqslant  3n$, that is $f(V)\geqslant \dfrac{n(2-\Delta)}{\Delta+1}$. Since $f$ be an arbitrary MRDF on $G$, we have found a lower bound for $\gamma_{MR}(G)$ as $\gamma_{MR}(G)\geqslant \dfrac{n(2-\Delta)}{\Delta+1}$

Since by  Theorem \ref{st}, we have $\gamma_{MR}(K_{1, n-1})=3-n$, this lower bound is sharp.
\end{proof}

\begin{theorem}
Let $T$ be a tree of order $n\geq 3$. The number of support vertices and  pendant vertices be $s$ and $l$, respectively. Then, $\gamma_{MR}(T)\leq \Big\lceil \dfrac{n+7s-5l}{4}\Big\rceil$. This bound is sharp if and only if $T \in \mathcal{T}$. 
\end{theorem}
\begin{proof}
Let $T$ be a tree that has  $s$ support vertices and $l$ pendant vertices. We proceed by induction on $n$. 
For $n=3$, we have $s=1, l=2$ and $\gamma_{MR}(T)=0$, and the equality holds.
This proves the base step of induction.

 Now, we suppose $n\geq 4$. 
 
 If  $diam(T)=2$, then $T$ is a star and we can obtain the result in terms of Theorem \ref{st}. So, we suppose that $diam(T) \geq 3$.
 
  Assume that for every tree $T^{'}$ with order $n^{'} <  n$, the inequality $\gamma_{MR}(T^{'}) \leq \lceil \dfrac{n^{'}+7s^{'}-5 l^{'}}{4}\rceil$ holds, that $s^{'}$ is the number of support vertices of $T^{'}$ and $l^{'}$ is the number of pendant vertices of $T^{'}$.

Now, if $T$ be a tree of order $n$ and  $diam(T)=d$. We consider a path $v_0v_1\cdots v_d$, with the length  $d$, in $T$. We have $\vert N(v_1)-\{v_2\} \vert=m \geq 1$.

The graph $T^{'}= T-(N(v)-\{v^{'}\})$ is a tree with the order  $n^{'}=n-m<n$. Also, we have $s'=s-1, l'= l-m+1$. So, based on the induction hypothesis, there is a function $f^{'}$ on $V(T')$ as a $\gamma_{MR}(T^{'})$-function that satisfies the inequality. By using the function $f^{'}$, We  define a function $f$ on $V(T)$ as follows,
$$
f(u)=\left\{\begin{array}{ccc}
f^{'}(u)&&u \in V(T'), \; u \neq v_1\\
2&&u= v_1\\
-1&&u\in N(v_1)- \{v_2\}
\end{array}\right.
 $$

According to the definition of function $f$, the inequality $f(N[u])\geq f'(N[u])$ holds for all vertices $u$ in $T^{'}$, except maybe $v_1$.
 
 Moreovere, we have $f(N[u])\geq 1$ for all vertices $u \in N(v_1)- \{v_2\}$. Therefore, $f$ is an MRDF on tree $T$.\\
So,

\begin{align*}
\gamma_{MR}(T)&\leq f(V(T)) \leq \gamma_{MR}(T')+3-m\\
&\leq \Big\lceil \dfrac{n'+7s'-5l'}{4}\Big\rceil +3-m\\
&\leq \Big\lceil \dfrac{n'+7s'-5l'}{4}+3-m\Big\rceil \\
&=  \Big\lceil \dfrac{(n-m)+7(s-1)-5(l-m+1)+12-4m}{4}\Big\rceil\\
&= \Big\lceil \dfrac{n+7s-5l}{4}\Big\rceil
\end{align*}

It is easy to see that the upper bound is sharp for any tree $T \in \mathcal{T}$.
 Since, we can consider $f$ as an MRDF on $T \in \mathcal{T}$ such that $f(v)=2$ for every support vertex and $f(u)=-1$ for any pendant vertex. Then $\gamma_{MR}(T)=2s-l= \Big\lceil \dfrac{n+7s-5l}{4}\Big\rceil$.
\end{proof}

A set $I\subseteq V(G)$ is called an independent set if the induced subgraph $G[I]$ has no edge. The size of the maximum independent set in $G$ is called the independence number of $G$, and is denoted by $\beta_0(G)$. 

\begin{proposition}
Let $T$ be a tree on $n\geq 2$ vertices. Then, $\gamma_{MR}(T) \leq 2n-3\beta_0(T)$. 
\end{proposition}

\begin{proof}
Let $T$ be a tree that the number of its support vertices is $s$ and the number of its pendant vertices is $l$. 
It is well known that $\beta_0(T)\geq \dfrac{n-s+l}{2}$ and $l \geq s$, and so we have $\beta_0 (T)\geq \dfrac{n}{2}$. \\
For an independent set $I=\{v_1, v_2, \cdots, v_{\beta_0(T)}\}$ of $T$ with cardinality $\beta_0 (T)$, let  $f:V(T)\to\{-1,1,2\}$ be a function such that $f(u)=-1$ for every vertex $u\in I$, and $f(u)=2$ otherwise. 

Obviously, $f$ is an MRDF on $T$ and $f(T)=2(v-\beta_0 (T))-\beta_0 (T)$. Thus, we deduce that  $\gamma_{MR}(T) \leq -\beta_0+2(n-\beta_0)= 2n-3\beta_0$, as required.

This bound is sharp for $T\in \mathcal{T}\cup\{P_2\}$. 
\end{proof}

\section{Some bounds for Majority Roman Domination number of graph $G \circ H$}

In this section, we obtain a lower bound on the majority roman domination number of every graph $G$ with order $n\geq 2$. 
Furthermore, we specify  lower and upper bounds on the majority roman domination number of $G \circ H$, where $H$ is a connected graph with $\delta \geq 2$.

The corona of two graphs $G_{1}$ and $G_{2}$, is the graph $G= G_{1} \circ G_{2}$ formed by one copy of $G_{1}$ and $|V(G_{1})|$ copies of $G_{2}$, where the i-th vertex of $G_{1}$ is adjacent to each vertex in the i-th copy of $G_{2}$.

\begin{figure}[H]
\begin{center}
\begin{tikzpicture}
\node at (0,1){
\begin{tikzpicture}
\node at (.3,1){$2$};\node at (-1,-.3){$2$};\node at (1,-.3){$2$};
\node at (1,3.3){$-1$};\node at (-1,3.3){$-1$};\node at (0,2.3){$1$};
\node at (-2.4,0){$-1$};\node at (-3,1.3){$-1$};\node at (-3,-1.3){$-1$};
\node at (2.4,0){$-1$};\node at (3,1.3){$-1$};\node at (3,-1.3){$-1$};
\draw(0,1)--(1,0);\draw(0,1)--(-1,0);\draw(0,1)--(0,2);\draw(0,1)--(-1,3);\draw(0,1)--(1,3);\draw(-1,0)--(1,0);
\draw(0,2)--(1,3);\draw(0,2)--(-1,3);\draw(-1,3)--(1,3);
\draw(-2,0)--(-3,1);\draw(-2,0)--(-3,-1);\draw(-3,1)--(-3,-1);\draw(-2,0)--(-1,0);\draw(-1,0)--(-3,1);\draw(-1,0)--(-3,-1);
\draw(2,0)--(3,1);\draw(2,0)--(3,-1);\draw(3,1)--(3,-1);\draw(2,0)--(1,0);\draw(1,0)--(3,-1);\draw(1,0)--(3,1);
\foreach \x in { -2, -1, 1, 2 }
\foreach \y in { 0 } \filldraw (\x,\y)circle(2pt);
\foreach \x in { -3, 3}
\foreach \y in { -1 } \filldraw (\x,\y)circle(2pt);
\foreach \x in { -3,0,3 }
\foreach \y in { 1 } \filldraw (\x,\y)circle(2pt);
\foreach \x in { 0}
\foreach \y in { 2 } \filldraw (\x,\y)circle(2pt);
\foreach \x in { -1,1 }
\foreach \y in { 3 } \filldraw (\x,\y)circle(2pt);
\end{tikzpicture}
};
\end{tikzpicture}
\end{center}
\caption{$\gamma_{MR}(K_3 \circ K_3)=-1$}
\label{fig8}
\end{figure}

\begin{example}\label{e1}
Let $G= H= K_3$. Then $\gamma_{MR}(G\circ H)=-1$.

In Figure 8, the graph $G\circ H$ is shown with an MRDF on it. The weight of this function is $-1$. In four step, we will show this function is a $\gamma_{MR}(K_3 \circ K_3)$-function.

Let $f$ be an MRDF on $G\circ H$. 

\textnormal{(i)} If $|V_{f,2}|\geq4$, then $f(V)\geq4\times2+8\times(-1)+0\times1=0$.

\textnormal{(ii)} If $|V_{f,2}|=1$, then $f(V)\geq1\times2+5\times(-1)+6\times1=3$.

\textnormal{(iii)} If $|V_{f,2}|=2$, then $f(V)\geq2\times2+7\times(-1)+3\times1=0$.\\
Hence, in each $\gamma_{MR}(K_3 \circ K_3)$-function, we must have:

\textnormal{(iv)} $|V_{f,2}|=3$. 
If $|V_{f,1}|=0$, then $|N_f|\leq4$, a contradiction.
\\
Consequently, $|V_{f,1}|\geq1$ and $f(V)\geq3\times2+8\times(-1)+1\times1=-1$
\\
According to the Figure \ref{fig8}, this proves $\gamma_{MR}(K_3 \circ  K_3)=-1$.
Note that  the $\gamma_{MR}(K_3 \circ K_3)$-function is not unique. In fact, if in figure \ref{fig8}, the value function $1$ is swapped by each of value functions $2$, the new function remains an MRDF on  $K_3 \circ K_3$.
\end{example}


\begin{theorem}
Let $G= K_{3k}$, $H=K_3$, and $k\geq 1$. Then, $\gamma_{MR}(G \circ H)=-k$.
\end{theorem}

\begin{proof}
According to the Example \ref{e1}, we have $\gamma_{MR}(K_3 \circ K_3)=-1$, this proves the result for $k=1$. 

Now, We assume that $k\geq2$ and construct an MRDF on $G \circ H$ with weight equal $-k$.

We set $V(G)=\{u_1, u_2, \cdots, u_{3k}\}$ and denote the copy of $H$ corresponding to vertex $u_i$ by $H_i$ with $V(H_i)=\{h_{i1}, h_{i2}, h_{i3}\}$, where $1\leq i\leq3k$.

First, we  define the following function.

$$
f(u)=\left\{\begin{array}{ccc}
2&& u\in V(G) \\
1&& u=h_{i1},\; 1\leq i \leq k\\
-1&& \text{otherwise}
\end{array}
\right.
$$
It is clear that each vertex with function value $-1$ is adjacent to a vertex with function value $2$. Moreover,
we have $f(N[u_i])=6k-1> 1$ for $1\leq i\leq k$ and $f(N[u_i])=6k-3> 1$ for $k+1\leq i\leq 3k$, and also $f(N[h_{ij}])= 1$ for $1\leq i\leq k$ and $1\leq j\leq3$. Thus $|N_f|=3k+3k=6k$ and Since $n=12k$, 
the function $f$ is an MRDF on $G\circ H$ with $f(V)=3k\times2+8k\times(-1)+k\times1=-k$. 

Now, suppose that $g$ be $\gamma_{MR}(G \circ H)$-function, that is, $g(V)\leq f(V)$.  
With the help of $g$, we will construct new functions, all of which are $\gamma_{MR}(G \circ H)$-function, and the last function will be $f$, as required.

\textbf{Claim 1.} We can suppose that $g(h_{ij})\ne2$, where $1\leq i\leq 3k$ and $1\leq j\leq3$.

If $g(h_{ij})=2$, for some $i$ and $j$, then there are two cases. 
In the first case, $g(u_i)=2$, by setting $g(h_{ij})=1$, we have a new MRDF with less weight, a contradiction. In the second case, $g(u_i)<2$, by setting $g(h_{ij})=g(u_i)$, and then $g(u_i)=2$, we have a new MRDF with the same weight. It is a $\gamma_{MR}(G \circ H)$-function.

\textbf{Claim 2.} We can suppose that $g(u_{i})=2$, where $1\leq i\leq 3k$.

If $g(u_r)<2$ for some $1\leq r\leq 3k$, then  $g(u_r)+g(h_{r1})+g(h_{r2})+g(h_{r3})\geq 2$ and by setting $g(u_r)=2, g(h_{r1})=1, g(h_{r2})=g(h_{r3})=-1$, we have a new MRDF by less weight, a contradiction.

\textbf{Claim 3.} We can suppose that $S_i=g(u_i)+g(h_{i1})+g(h_{i2})+g(h_{i3})= -1$ or $1$, where $1\leq i\leq 3k$.

For each $i$, since $g(u_i)=2$ and $g(h_{ij})<2$, where $1\leq j\leq3$, then 
$S_i\in\{-1, 1, 3 ,5\}$. If $S_i=3$ or $5$, then by setting $g(h_{r1})=1, g(h_{r2})=g(h_{r3})=-1$, we have a new MRDF with less weight, a contradiction.

\textbf{Claim 4.} $|\{i:S_i=1,1\leq i\leq3k\}|= k$

If $|\{i:S_i=1,1\leq i\leq3k\}|< k$, then $|N_g|<3k+3|V(G)|=6k<\dfrac{n}{2}$, a contradiction. Also, if $|\{i:S_i=1,1\leq i\leq3k\}|> k$, then $g(V)>k\times1-2k\times(-1)=-k=f(V)$, a contradiction.

By this four claims, we have $g(V)=k-2k=-k$, and we have proved $\gamma_{MR}(G \circ H)=\gamma_{MR}(K_{3k} \circ K_3)=-k$
\end{proof}

\begin{lemma}\label{t1}
Let $m, ~n \in N$ and $m\geq 3$. Then $\lfloor \dfrac{n}{m}\rfloor m+n \leq \lceil \dfrac{nm+n}{2}\rceil$.
\end{lemma}

\begin{proof}
It is clear that $\lfloor\dfrac{n}{m}\rfloor \leq \dfrac{n}{m}$. So, $\lfloor\dfrac{n}{m}\rfloor m+n\leq 2n$. On the other hand, since $m\geq 3$, we have  $2n \leq \dfrac{mn+n}{2}$. Therefore, $\lfloor \dfrac{n}{m}\rfloor m+n \leq \lceil \dfrac{nm+n}{2}\rceil$.
\end{proof}
\begin{theorem}
For every graph $G$, $\vert G\vert= n$, and connected graph $H$ with $\delta (H)\geq 2$, $\vert H\vert=m$, we have $\gamma_{MR}(G\circ H)\leq (m-4)\lceil \dfrac{n}{2}\rceil -m \lfloor \dfrac{n}{2}\rfloor +2n$. 
\end{theorem}
\begin{proof}
Let $V(G)= \{ v_1, v_2, \cdots, v_n\}$. We consider graphs $H_{i}$ with $V(H_i)= \{h_{i1}, h_{i2}, h_{i3},\\ \cdots, h_{im}\}$ for every vertex $v_i \in V(G)$, $ 1\leq i \leq n$, such that every vertex of $H_{i}$  is adjacent to $v_i$. Now, we define the function $f$ as follows.
$$
f(u)=\left\{\begin{array}{ccl}
2&& u\in V(G) \\
1&& u=h_{ij} \in H_{i}, 1\leq i \leq \lceil  \dfrac{n}{2}\rceil, 1\leq j\leq m-1\\
-1&& \text{otherwise}
\end{array}
\right.
$$
For every vertex of $h_{ij} \in H_{i}$ that $f(h_{ij})=-1$, there is the vertex $v_i$ in graph $G$ such that this vertex is adjacent to $h_{ij}$ and $f(v_i)= 2$. Moreover, for $m \lceil \frac{n}{2}\rceil+ \lceil \frac{n}{2} \rceil$ vertices of $G\circ H$, like $v$, $f(N[v])\geq 1$. Now, the function $f$ is an MRDF of $G\circ H$. Therefore, $\gamma_{MR}(G\circ H)\leq (m-4)\lceil \frac{n}{2}\rceil -m \lfloor \frac{n}{2}\rfloor +2n$.
 \\
This bound is sharp for $K_4$, since, $\gamma_{MR}(K_4=K_1\circ K_3)=1\leq (3-4)\lceil \dfrac{1}{2}\rceil -3 \lfloor \dfrac{1}{2}\rfloor + 2 \times 1= 1$
\end{proof}

\begin{theorem}
Let $G$ be a graph with $\vert G\vert= n$, and $H$ be a connected graph with $\delta (H)\geq 2$ and $\vert H\vert=m$. Then $(2-m)n +2 \lfloor \dfrac{n}{m}\rfloor \leq \gamma_{MR}(G\circ H) $. 
\end{theorem}
\begin{proof}
Let $V(G)= \{ v_1, v_2, \cdots, v_n\}$. We consider graph $H_{i}$ with $V(H_i)= \{h_{i1}, h_{i2}, h_{i3},\\ \cdots, h_{im}\}$ for every vertex $v_i \in V(G)$, $ 1\leq i \leq n$ such that every vertex of $H_{i}$  is adjacent to $v_i$. Now, we define function $f$ as follows.

$$
f(u)=\left\{\begin{array}{ccl}
2&& u\in V(G) \\
1&& u=h_{ij} \in H_{i}, 1\leq i \leq \lceil  \dfrac{n}{m}\rceil,\\
-1&& \text{otherwise}
\end{array}
\right.
$$
For every vertex of $u \in H_{i}$ with $f(u)=-1$, there is a vertex in $G$, $v_i$, such that this vertex is adjacent $u$ and $f(v_i)= 2$. Moreover, according to the Lemma \ref{t1} and the definition of $f$, for $\lfloor \dfrac{n}{m}\rfloor m+n \leq \lceil \dfrac{nm+n}{2}\rceil$ vertices of the $G \circ H$, $v$, $f(N[v])\geq 1$. If $\lfloor \dfrac{n}{m}\rfloor m+n = \lceil \dfrac{nm+n}{2} \rceil$, then $f$ is a majority Roman dominating function of $G \circ H$, and $(2-m)n +2 \lfloor \dfrac{n}{m}\rfloor = \gamma_{MR}(G \circ H) $. Otherwise, ($\lfloor \dfrac{n}{m} \rfloor m+n < \lceil \dfrac{nm+n}{2}\rceil$), then $f$ can not be a majority Roman dominating function of $G \circ H$, so, this proves the result.\\ 
Furthermore, in terms of Example \ref{e1} this bound is sharp, $\gamma_{MR}(K_{3n}\circ K_3)= (2-3) 3n+2 \lfloor \dfrac{3n}{3}\rfloor=-n $.
\end{proof}

\end{document}